\documentclass[12pt,a4paper]{article}
\usepackage{amsmath,amsthm,amssymb, mathrsfs}
\usepackage{hyperref}
\input xy
\xyoption{all}
\usepackage{epsfig}

\newtheorem{thm}{Theorem}[section]

\newtheorem{lem}[thm]{Lemma}

 \theoremstyle{definition}

\theoremstyle{remark}
\newtheorem{remark}[thm]{Remark}

\numberwithin{equation}{section}

\newcommand{\ben}{\begin{enumerate}}

\newcommand{\een}{\end{enumerate}}
\newcommand{\bit}{\begin{itemize}}
\newcommand{\eit}{\end{itemize}}

\newcommand{\h}{\hskip 0.1cm}
\newcommand{\hh}{\hskip 0.6cm}

\begin{document}

\title
{An inversion formula for the spherical mean transform with data on an ellipsoid in
two and three dimensions}

\vskip 10cm
\author{Yehonatan Salman \\ Bar-Ilan University}
\date{}

\maketitle

\begin{abstract}
In the articles [1] and [2] of D. Finch, M. Haltmeier, S. Patch and D. Rakesh inversion formulas
were found in any dimension $n\geq2$ for recovering a smooth function with compact support in the unit ball
from spherical means centered on the unit sphere. The aim of this article is to show that the methods
used in [1], [2] can be modified in order to get similar inversion formulas from spherical means centered on an ellipsoid
in two and three dimensional spaces.

\end{abstract}

\section{Introduction and mathematical background}
\hh Let $n$ be an integer equals to 2 or 3. Denote by $S^{n - 1} = \{x\in\Bbb R^{n}:|x| = 1\}$ the unit sphere in $\Bbb R^{n}$. For a closed smooth bounded domain $\Omega$ denote by $\partial\Omega$ its boundary and by $C_{0}^{\infty}(\Omega)$ the class of infinitely differentiable functions defined on $\Bbb R^{n}$ which vanish outside $\Omega$. For any $n$ positive real numbers $a_{1},..., a_{n} > 0$ denote by $E = E_{a_{1},...,a_{n}}$ the standard solid ellipsoid in $\Bbb R^{n}$ with axes $a_{1},..., a_{n}$, that is $E = \{x\in\Bbb R^{n}:a_{1}^{-2}x_{1}^{2} + ... + a_{n}^{-2}x_{n}^{2} \leq 1\}$. Let $dS_{E}$ be the measure on $\partial E$ corresponding to the measure that comes from integration on the unit sphere, that is, in the two dimensional case we have
$$\int_{\partial E}f(p)dS_{E}(p)=\int_{-\pi}^{\pi}f(a_{1}\cos\theta,a_{2}\sin\theta)d\theta,$$
and in the three dimensional case
$$\int_{\partial E}f(p)dS_{E}(p) =
\int_{0}^{\pi}\int_{-\pi}^{\pi}f(a_{1}\sin\phi \cos\theta, a_{2}\sin\phi \sin\theta, a_{3}\cos\phi)\sin\phi d\theta d\phi.$$
The spherical mean transform $R$ takes any function $f\in C^{\infty}(\Bbb R^{n})$ to \vskip -0.3cm
$$(Rf)(x,r) = \int_{S^{n - 1}}f(x + r\omega)dS(\omega)$$
where $dS$ is the standard surface measure on $S^{n - 1}$.\\

The goal of this article is to recover the function $f\in C_{0}^{\infty}(E)$ from the spherical mean transform $Rf$ restricted to the surface $\partial E\times[0,\infty)$ where $E = E_{a_{1},...,a_{n}}$ and $a_{1},..., a_{n}$ are positive real numbers. That is, our goal is to recover $f$ from spherical means with arbitrary radius centered on $\partial E$. We will give inversion formulas that recover $f$ in the two and three dimensional case. Inversion formulas for the general case of an arbitrary ellipsoid in two or three dimensional space can be easily obtained from inversion formulas for the standard ellipsoid by using translations and rotations. The methods presented in this article are based on and generalize the methods presented in [1], [2].\\

The inversion problem for the spherical mean transform has been studied deeply recently due to its applications in
various mathematical, physical and scientific fields such as thermo and photoacoustic tomography, radar and sonar imaging, approximation theory and other areas ([6, 8, 11, 12, 14]). The inversion problem has been found out to be extremely important in thermoacoustic tomogrpahy (TAT for short). In TAT a short duration EM pulse (electro magnetic pulse) is sent through a biological object $\Omega$ which causes a thermoacoustic response from $\Omega$ resulting in a propagation wave $G(x,t)$ which describes the amount of energy absorbed at location $x$ and time $t\geq 0$. The transducers (the receivers that measure the pressure wave $G(x,t)$) are placed on the boundary $\partial\Omega$ of the biological object and from this data one has to recover the initial pressure wave $f(x) = G(x,0)$ on $\Omega$. Recovering $f$ can be useful in finding cancerous cells in the biological object $\Omega$ since they absorb several times more energy than ordinary cells ([6, 15]). The recovering of the initial pressure wave is strongly related to the inversion problem for the spherical mean transform since the amount of EM energy absorbed at location $x$ and time $t\geq0$ is the average energy that was generated on the sphere with center $x$ and radius $t$ at the initial time which is just the spherical mean transform $(Rf)(x,t)$ of the initial pressure wave.

\section{Main result and known results}

\hh From now on $a_{1},...,a_{n}$ will be fixed positive real numbers. The inversion formulas presented in this article differ in the two and three dimensional case. Both formulas are modifications for the elliptic case of Theorem 1.1 in [1] for the two dimensional case and of Theorem 3 in [2] for the three dimensional case. In the two dimensional case we have the following theorem
\begin{thm}
Let $E = E_{a_{1},a_{2}}\subset\Bbb R^{2}$ be the standard solid ellipsoid with axes $a_{1}, a_{2}$ and $f\in C_{0}^{\infty}(E)$, then for every $x\in E$,

\begin{equation}
f(x) = \left(\frac{1}{a_{1}^{2}}\frac{\partial^{2}}{\partial x_{1}^{2}} + \frac{1}{a_{2}^{2}}\frac{\partial^{2}}{\partial x_{2}^{2}}\right)\frac{1}{4\pi^{2}}\int_{\partial E}\int_{0}^{\infty}r(Rf)(p,r)\log\left|r^{2} - |x - p|^{2}\right|drdS_{E}(p).
\end{equation}

\end{thm}
In the three dimensional case we have the following theorem

\begin{thm}
Let $E = E_{a_{1},a_{2},a_{3}}\subset\Bbb R^{3}$ be the standard solid ellipsoid with axes $a_{1}, a_{2}, a_{3}$ and $f\in C_{0}^{\infty}(E)$, then for every $x\in E$,
\begin{equation}
f(x) = -\left(\frac{1}{a_{1}^{2}}\frac{\partial^{2}}{\partial x_{1}^{2}} + \frac{1}{a_{2}^{2}}\frac{\partial^{2}}{\partial x_{2}^{2}} + \frac{1}{a_{3}^{3}}\frac{\partial^{2}}{\partial x_{3}^{2}}\right)\frac{1}{4\pi^{2}}\int_{\partial E}\int_{0}^{\infty}r^{2}(Rf)(p,r)\delta\left|r^{2} - |x - p|^{2}\right|drdS_{E}(p),
\end{equation}
where $\delta$ is the Dirac delta function.
\end{thm}

It should be noted that the integrals that appear in the theorems above are well defined since $f\in C_{0}^{\infty}(E)$.

Inversion formulas for the spherical mean transform with centers on a sphere have been derived recently in [1, 2, 7, 10, 16]. In [2] an inversion formula was found for odd dimensions and later in [1] a similar inversion formula was found for even dimensions. In [7] a uniform inversion formula was found for any dimension.

Inversion formulas with spheres centered on an ellipsoid have been derived quite recently ([4, 5, 9, 13]). In [9] an inversion formula was found for the two dimensional case with spheres centered on the boundary of an arbitrary smooth domain, the inversion formula recovers $f$ modulo an integral operator which vanishes for elliptical domains. In [5] a generalization of the results obtained in [9] was found for any dimension. A different inversion formula for elliptical domains for any dimension was found in [13]. The methods for recovering $f$ presented in this article are different from the methods in [4, 5, 9, 13] and are based on the methods presented in [1] ,[2]. In fact the methods presented here are modified versions for elliptical domains of the methods in [1], [2].

\section{Proofs}
\hh We first prove Theorem 2.1. The proof of Theorem 2.1 relies on the following lemma which is a modification of Proposition 2.1 in [1] for the elliptic case.

\begin{lem}
Let $E = E_{a_{1},a_{2}}\subset\Bbb R^{2}$ be the standard solid ellipsoid with axes $a_{1}, a_{2}$. For every $x,y\in E, x \neq y$ we have
$$\int_{\partial E}\log\left||x - p|^{2} - |y - p|^{2}\right|dS_{E}(p) = 2\pi \log|A(x - y)| + C$$
where $A =
\left(
  \begin{array}{cc}
    a_{1} & 0 \\
    0 & a_{2} \\
  \end{array}
\right)$
and $C$ is a constant independent of $x$ and $y$.
\end{lem}

\begin{proof}
Start with the identity
$$ \left||x - p|^{2} - |y - p|^{2}\right| = 2\left|\left(\frac{x + y}{2} - p\right)\cdot (x - y)\right|$$
which is true for every $x, y, p \in \Bbb R^{2}$. Taking the $\log$ function on both sides of the last equation yields
$$\log\left||x - p|^{2} - |y - p|^{2}\right| = \log2 +
  \log\left|\left(\frac{x + y}{2} - p\right)\cdot (x - y)\right| = $$
\begin{equation}
\log2 + \log\left|\frac{1}{2}(|x|^{2} - |y|^{2}) - p\cdot (x - y)\right|.
\end{equation}
Let us represent $p$ in the form $p = (a_{1}\cos\theta, a_{2}\sin\theta)$, then
$$p\cdot (x - y) = (a_{1}\cos\theta, a_{2}\sin\theta)\cdot (x_{1} - y_{1}, x_{2} - y_{2}) =
a_{1}(x_{1} - y_{1})\cos\theta + a_{2}(x_{2} - y_{2})\sin\theta = $$
$$(\cos\theta, \sin\theta)\cdot (a_{1}(x_{1} - y_{1}), a_{2}(x_{2} - y_{2})) = $$
\begin{equation}
= (\cos\theta, \sin\theta)\cdot \left(
  \begin{array}{cc}
    a_{1} & 0 \\
    0 & a_{2} \\
  \end{array}
\right)(x_{1} - y_{1}, x_{2} - y_{2}) = e^{i\theta}\cdot A(x - y)
\end{equation}
Here $\cdot$ denotes the dot product. Hence from (3.1) and (3.2) we have
$$\log\left||x - p|^{2} - |y - p|^{2}\right| = \log2 + \log\left|\frac{1}{2}(|x|^{2} - |y|^{2}) - e^{i\theta}\cdot A(x - y)\right| $$
$$ = \log2 + \log|A(x - y)| + \log\left|\frac{|x|^{2} - |y|^{2}}{2|A(x - y)|} - e^{i\theta}\cdot\frac{A(x - y)}{|A(x - y)|}\right|.$$
Integrating both sides of the last equation on $\partial E$ with measure $dS_{E}(p)$ yields
$$\int_{\partial E} \log\left||x - p|^{2} - |y - p|^{2}\right|dS_{E}(p) = $$
$$ = 2\pi \log2 + 2\pi \log|A(x - y)| +
\int_{-\pi}^{\pi}\log\left|\frac{|x|^{2} - |y|^{2}}{2|A(x - y)|} - e^{i\theta}\cdot\frac{A(x - y)}{|A(x - y)|}\right|d\theta.$$
Denoting $e^{i\psi} = A(x - y)/|A(x - y)|$, using the change of variables $\theta \mapsto \theta + \psi$ and the
$2\pi$ periodicity of the function $\theta \mapsto e^{i\theta} \cdot e^{i\psi} = \cos(\theta - \psi)$ gives
$$\int_{\partial E} \log\left||x - p|^{2} - |y - p|^{2}\right|dS_{E}(p) = $$
$$2\pi \log2 + 2\pi \log|A(x - y)| +
\int_{-\pi}^{\pi}\log\left|\frac{|x|^{2} - |y|^{2}}{2|A(x - y)|} - \cos\theta\right|d\theta. $$
Let us note that
$$\left|\frac{|x|^{2} - |y|^{2}}{2|A(x - y)|}\right| = \frac{\left|A^{-1}(x + y)\cdot A(x - y)\right|}{2|A(x - y)|}$$
$$ \leq \frac{1}{2}|A^{- 1}(x + y)| \leq \frac{1}{2}|A^{- 1}x| + \frac{1}{2}|A^{- 1}y| = $$
$$\frac{1}{2}\left|\left(\frac{x_{1}}{a_{1}},\frac{x_{2}}{a_{2}}\right)\right| +
  \frac{1}{2}\left|\left(\frac{y_{1}}{a_{1}},\frac{y_{2}}{a_{2}}\right)\right| = $$
$$\frac{1}{2}\sqrt{\frac{x_{1}^{2}}{a_{1}^{2}} + \frac{x_{2}^{2}}{a_{2}^{2}}} +
  \frac{1}{2}\sqrt{\frac{y_{1}^{2}}{a_{1}^{2}} + \frac{y_{2}^{2}}{a_{2}^{2}}} \leq 1.$$
Hence we can write
$$\frac{|x|^{2} - |y|^{2}}{2|A(x - y)|} = \cos\alpha$$
and thus
$$\int_{\partial E} \log\left||x - p|^{2} - |y - p|^{2}\right|dS_{E}(p) = $$
$$2\pi \log2 + 2\pi \log|A(x - y)| +
\int_{-\pi}^{\pi}\log\left|\cos\alpha - \cos\theta\right|d\theta = $$
$$ 2\pi \log2 + 2\pi \log|A(x - y)| +
\int_{-\pi}^{\pi}\log\left|\frac{1}{2}\sin\left(\frac{\theta + \alpha}{2}\right)\sin\left(\frac{\theta - \alpha}{2}\right)
 \right|d\theta = $$
$$ 2\pi \log|A(x - y)| +
\int_{-\pi}^{\pi}\log\left|\sin\left(\frac{\theta + \alpha}{2}\right)\right|d\theta +
\int_{-\pi}^{\pi}\log\left|\sin\left(\frac{\theta - \alpha}{2}\right)\right|d\theta.$$
Using the change of variables $t = \theta + \alpha$ and $t = \theta - \alpha$ in the first and the second
integrals respectively, and observing that the function $t\mapsto |\sin(t/2)|$ is $2\pi$ periodic, we have
$$\int_{\partial E} \log\left||x - p|^{2} - |y - p|^{2}\right|dS_{E}(p) = $$
$$2\pi \log|A(x - y)| + 2\int_{-\pi}^{\pi}\log\left|\sin\left(\frac{t}{2}\right)\right|dt.$$
The important thing to note is that $2\int_{-\pi}^{\pi}\log\left|\sin\left(t/2)\right)\right|dt$ is a constant $C$ which
does not depend on $x$ or $y$, hence
$$\int_{\partial E} \log\left||x - p|^{2} - |y - p|^{2}\right|dS_{E}(p) = $$
$$2\pi \log|A(x - y)| + C.$$
This proves the lemma.
\end{proof}

We now turn to the proof of Theorem 2.1 which is a modification of the proof of Theorem 1.1 in [1].

\textbf{Proof of Theorem 2.1} :
For $f\in C_{0}^{\infty}(E)$ and $x\in E$ let us compute
$$\int_{\partial E}\int_{0}^{\infty}r(Rf)(p,r)\log|r^{2} - |x - p|^{2}|drdS_{E}(p) = $$
$$\int_{\partial E}\int_{\Bbb R^{2}}f(y)\log||y - p|^{2} - |x - p|^{2}|dydS_{E}(p) = $$
$$\int_{\Bbb R^{2}}f(y)\int_{\partial E}\log||y - p|^{2} - |x - p|^{2}|dS_{E}(p)dy. $$
Using Lemma 3.1 on the inner integral in the right hand side of the last equality gives
$$\int_{\partial E}\int_{0}^{\infty}r(Rf)(p,r)\log|r^{2} - |x - p|^{2}|drdS_{E}(p) = $$
$$\int_{\Bbb R^{2}}f(y)(2\pi \log|A(x - y)| + C)dy.$$
Hence by applying the operator $a_{1}^{-2}\partial^{2}/\partial x_{1}^{2} + a_{2}^{-2}\partial^{2}/\partial x_{2}^{2}$ on both sides of the last equality we get
$$\left(\frac{1}{a_{1}^{2}}\frac{\partial^{2}}{\partial x_{1}^{2}} + \frac{1}{a_{2}^{2}}\frac{\partial^{2}}{\partial x_{2}^{2}}\right)\int_{\partial E}\int_{0}^{\infty}r(Rf)(p,r)\log|r^{2} - |x - p|^{2}|drds(p) = $$
$$\left(\frac{1}{a_{1}^{2}}\frac{\partial^{2}}{\partial x_{1}^{2}} + \frac{1}{a_{2}^{2}}\frac{\partial^{2}}{\partial x_{2}^{2}}\right)\int_{\Bbb R^{2}}f(y)(2\pi \log|A(x - y)| + C)dy = $$
$$2\pi\int_{\Bbb R^{2}}f(y)\left(\frac{1}{a_{1}^{2}}\frac{\partial^{2}}{\partial x_{1}^{2}} + \frac{1}{a_{2}^{2}}\frac{\partial^{2}}{\partial x_{2}^{2}}\right)\log|a_{1}(x_{1} - y_{1}), a_{2}(x_{2} - y_{2})|dy = $$
$$2\pi\int_{\Bbb R^{2}}f(y)(\partial_{1}^{2} + \partial_{2}^{2})\log|a_{1}(x_{1} - y_{1}), a_{2}(x_{2} - y_{2})|dy = $$
$$2\pi\int_{\Bbb R^{2}}f(y)2\pi\delta(a_{1}(x_{1} - y_{1}), a_{2}(x_{2} - y_{2}))dy  = 4\pi^{2} f(x),$$
where we have used the fact that $(2\pi)^{- 1}\log|x|$ is a fundamental solution of the Laplacian in $\Bbb R^{2}$. This proves formula (2.1) and thus Theorem 2.1.

We now turn to the proof of Theorem 2.2. The proof of Theorem 2.2 relies on Lemma 3.2 which will be stated below. The method of the proof of Lemma 3.2 is taken from the first part of the proof of Theorem 3 in [2].

\begin{lem}
Let $E = E_{a_{1},a_{2},a_{3}}\subset\Bbb R^{3}$ be the standard solid ellipsoid with axes $a_{1}, a_{2}, a_{3}$. For every $x,y\in E,\h x \neq y$ we have
$$\int_{\partial E}\delta\left||x - p|^{2} - |y - p|^{2}\right|dS_{E}(p) = \frac{\pi}{|A(x - y)|}$$
where $A = \left(
             \begin{array}{ccc}
               a_{1} & 0 & 0 \\
               0 & a_{2} & 0 \\
               0 & 0 & a_{3} \\
             \end{array}
           \right)$.
\end{lem}

\begin{proof}
By the coarea formula, if $M$ is a two dimensional surface in $\Bbb R^{3}$, given by $\phi(p) = 0$, with
$\nabla\phi(p)\neq 0$ at every point of $M$, then for every $h\in C(\Bbb R^{3})$
\begin{equation}
\int_{M}h(p)dS_{p} = \int_{\Bbb R^{3}}h(p)|\nabla\phi(p)|\delta(\phi(p))dp
\end{equation}
where $dS_{p}$ is a surface element of $M$. In [3] Chapter V it was proved that if $\partial E$ is given by $\phi(p) = a_{1}^{- 2}p_{1}^{2} + a_{2}^{- 2}p_{2}^{2} + a_{3}^{- 2}p_{3}^{2} - 1 = 0$, $dS_{p}$ is a surface element of $\partial E$ and $dS_{E}$ is the measure on $\partial E$ coming from the unit sphere as defined above, then the following relation holds
$$dS_{E}(p)= \frac{2}{a_{1}a_{2}a_{3}|\nabla\phi(p)|}dS_{p}.$$
Using this and (3.3) with $M = \partial E$ yields
$$\int_{\partial E}h(p)dS_{E}(p) = \frac{2}{a_{1}a_{2}a_{3}}\int_{\partial E}\frac{h(p)}{|\nabla\phi(p)|}dS_{p} =
\frac{2}{a_{1}a_{2}a_{3}}\int_{\Bbb R^{3}}h(p)\delta(\phi(p))dp.$$
Hence in particular for $h(p) = \delta\left||x - p|^{2} - |y - p|^{2}\right|$ we have
$$\int_{\partial E}\delta\left||x - p|^{2} - |y - p|^{2}\right|dS_{E}(p) = $$
$$\frac{2}{a_{1}a_{2}a_{3}}\int_{\Bbb R^{3}}\delta\left||x - p|^{2} - |y - p|^{2}\right|\delta(a_{1}^{- 2}p_{1}^{2} + a_{2}^{- 2}p_{2}^{2} + a_{3}^{- 2}p_{3}^{2} - 1)dp = $$
$$\frac{2}{a_{1}a_{2}a_{3}}\int_{\Bbb R^{3}}\delta\left||x|^{2} - |y|^{2} - 2p\cdot (x - y)\right|\delta(a_{1}^{- 2}p_{1}^{2} + a_{2}^{- 2}p_{2}^{2} + a_{3}^{- 2}p_{3}^{2} - 1)dp. $$
The change of variables $\lambda_{1} = a_{1}^{- 1}p_{1},\h \lambda_{2} = a_{2}^{- 1}p_{2},\h \lambda_{3} = a_{3}^{- 1}p_{3} $ gives
$$\int_{\partial E}\delta\left||x - p|^{2} - |y - p|^{2}\right|dS_{E}(p) =
2\int_{\Bbb R^{3}}\delta\left||x|^{2} - |y|^{2} - 2\lambda\cdot (A(x - y))\right|\delta(|\lambda|^{2} - 1)d\lambda $$
where $\lambda = (\lambda_{1}, \lambda_{2}, \lambda_{3})$ and $A = \left(
             \begin{array}{ccc}
               a_{1} & 0 & 0 \\
               0 & a_{2} & 0 \\
               0 & 0 & a_{3} \\
             \end{array}
           \right).$
Let $Q$ be an orthogonal matrix satisfying $QA(x - y) = |A(x - y)|e_{3}$ where $e_{3} = (0,0,1)$. Using the change of variables $\lambda = Q^{- 1}p$ gives
$$\int_{\partial E}\delta\left||x - p|^{2} - |y - p|^{2}\right|dS_{E}(p) =
2\int_{\Bbb R^{3}}\delta\left||x|^{2} - |y|^{2} - 2p\cdot (QA(x - y))\right|\delta(|p|^{2} - 1)dp =$$
\begin{equation}
2\int_{\Bbb R^{3}}\delta\left||x|^{2} - |y|^{2} - 2|A(x - y)|p_{3}\right|\delta(|p|^{2} - 1)dp .
\end{equation}
Using (3.3) on equation (3.4) with the surface $H$ given by
$\phi(p) = |x|^{2} - |y|^{2} - 2|A(x - y)|p_{3} = 0$ and noting that $|\nabla\phi(p)| = 2|A(x - y)|$ gives
$$\int_{\partial E}\delta\left||x - p|^{2} - |y - p|^{2}\right|dS_{E}(p) =
\frac{1}{|A(x - y)|}\int_{H}\delta(|p|^{2} - 1)dH.$$
using the following parametrization for $H$
$$H : (p_{1}, p_{2}, a), \h p_{1}, p_{2}\in \Bbb R,\h a = \frac{|x|^{2} - |y|^{2}}{2|A(x - y)|},$$
gives
\begin{equation}
\int_{\partial E}\delta\left||x - p|^{2} - |y - p|^{2}\right|dS_{E}(p) =
\frac{1}{|A(x - y)|}\int_{\Bbb R^{2}}\delta(p_{1}^{2} + p_{2}^{2} + |a|^2 - 1)dp_{1}dp_{2}.
\end{equation}
So far we used (3.3) only for two dimensional surfaces but in fact formula (3.3) is true in any dimension and in particular it is holds for one dimensional surfaces. Hence using (3.3) on equation (3.5) for the circle $C : \phi(p) = p_{1}^{2} + p_{2}^{2} + |a|^2 - 1 = 0$ and observing that $|\nabla\phi(p)| = \sqrt{4p_{1}^{2} + 4p_{2}^{2}} = 2\sqrt{1 - |a|^2}$ gives
$$\int_{\partial E}\delta\left||x - p|^{2} - |y - p|^{2}\right|dS_{E}(p) =
\frac{1}{|A(x - y)|}\int_{C}\frac{dC}{2\sqrt{1 - |a|^2}} = \frac{\pi}{|A(x - y)|}. $$
This proves Lemma 3.2.
\end{proof}

\begin{remark}
As in the proof of Lemma 3.1 it can be shown that $|a| < 1$ and thus the square root $\sqrt{1 - |a|^{2}}$ is defined well.
\end{remark}

We now turn to the proof of Theorem 2.2 which is a modification of the first part of the proof of Theorem 3 in [2].

\textbf{Proof of Theorem 2.2} :
For $f\in C_{0}^{\infty}(E)$ and $x\in E$ let us compute
$$\int_{\partial E}\int_{0}^{\infty}r^{2}(Rf)(p,r)\delta\left|r^{2} - |x - p|^{2}\right|drdS_{E}(p) = $$
$$\int_{\partial E}\int_{\Bbb R^{3}}f(y)\delta\left||y - p|^{2} - |x - p|^{2}\right|dydS_{E}(p) = $$
$$\int_{\Bbb R^{3}}f(y)\int_{\partial E}\delta\left||y - p|^{2} - |x - p|^{2}\right|dS_{E}(p)dy. $$
Using Lemma 3.2 on the inner integral in the right hand side of the last equality gives
$$\int_{\partial E}\int_{0}^{\infty}r^{2}(Rf)(p,r)\delta\left|r^{2} - |x - p|^{2}\right|drdS_{E}(p) =
\int_{\Bbb R^{3}}f(y)\frac{\pi}{|A(x - y)|}dy. $$
Hence by applying the operator $a_{1}^{- 2}\partial^{2}/\partial x_{1}^{2} + a_{2}^{- 2}\partial^{2}/\partial x_{2}^{2} + a_{3}^{- 3}\partial^{2}/\partial x_{3}^{2}$ on both sides of the last equality we get
$$\left(\frac{1}{a_{1}^{2}}\frac{\partial^{2}}{\partial x_{1}^{2}} + \frac{1}{a_{2}^{2}}\frac{\partial^{2}}{\partial x_{2}^{2}} + \frac{1}{a_{3}^{2}}\frac{\partial^{2}}{\partial x_{3}^{2}}\right)\int_{\partial E}\int_{0}^{\infty}r^{2}(Rf)(p,r)\delta\left|r^{2} - |x - p|^{2}\right|drdS_{E}(p) = $$
$$ \left(\frac{1}{a_{1}^{2}}\frac{\partial^{2}}{\partial x_{1}^{2}} + \frac{1}{a_{2}^{2}}\frac{\partial^{2}}{\partial x_{2}^{2}} + \frac{1}{a_{3}^{2}}\frac{\partial^{2}}{\partial x_{3}^{2}}\right)\int_{\Bbb R^{3}}f(y)\frac{\pi}{|A(x - y)|}dy = $$
$$ \left(\frac{1}{a_{1}^{2}}\frac{\partial^{2}}{\partial x_{1}^{2}} + \frac{1}{a_{2}^{2}}\frac{\partial^{2}}{\partial x_{2}^{2}} + \frac{1}{a_{3}^{2}}\frac{\partial^{2}}{\partial x_{3}^{2}}\right)\int_{\Bbb R^{3}}f(y)\frac{\pi}{|a_{1}(x_{1} - y_{1}), a_{2}(x_{2} - y_{2}), a_{3}(x_{3} - y_{3})|}dy = $$
$$\int_{\Bbb R^{3}}f(y)(\partial_{1}^{2} + \partial_{2}^{2} + \partial_{3}^{2})\frac{\pi}{|a_{1}(x_{1} - y_{1}), a_{2}(x_{2} - y_{2}), a_{3}(x_{3} - y_{3})|}dy = $$
$$-4\pi^{2}\int_{\Bbb R^{3}}f(y)\delta(a_{1}(x_{1} - y_{1}), a_{2}(x_{2} - y_{2}), a_{3}(x_{3} - y_{3}))dy = -4\pi^{2} f(x),$$
where we have used that fact that $(- 4\pi|x - y|)^{- 1}$ is a fundamental solution for the Laplacian in $\Bbb R^{3}$. This proves formula (2.2) and thus Theorem 2.2.

\section*{Acknowledgments}
\hh I would like to thank my supervisor Professor Mark Agranovsky, Bar-Ilan University, for useful advices, discussions and remarks and Professor Leonid Kunyansky, University of Arizona, for checking the results of this article and useful remarks.

\end{document}